\documentclass[10pt]{amsart}
\usepackage[a4paper, total={5.5in, 8in}]{geometry}

\usepackage[dvipsnames]{xcolor} %% Julien
\usepackage[urlcolor=purple,linkcolor=blue, citecolor=Emerald, colorlinks=true, linktocpage=true]{hyperref}
\usepackage{tikz}
\usepackage{graphicx,amssymb,eucal,mathrsfs}
\usepackage{latexsym,amsmath,epsfig,epic,eepic}
\usepackage{amscd}
\usepackage{amsthm}
\usepackage{mathabx}
\usepackage{indentfirst}
\usepackage{epsf,graphicx,pgf}
\usepackage{pstricks,mathrsfs}
\usepackage{mathtools}
\usepackage{tikz-cd}
\usepackage{comment}
\usepackage{stmaryrd}
\usepackage[foot]{amsaddr}
\usepackage{enumitem}

\title[A non-autonomous Hamiltonian diffeomorphism with roots]{A non-autonomous Hamiltonian diffeomorphism with roots of all orders}
\author{Nicolas Grunder}
\email{nicolas.grunder@unine.ch}

\author{Baptiste Serraille}
\address{Department of Mathematics, ETH-Z\"{u}rich, R\"{a}mistrasse 101, 8092 Z\"{u}rich, Switzerland.}
\email{baptiste.serraille@math.ethz.ch}

\newtheorem{thm}{Theorem}
\newtheorem{thmA}{Theorem}

\newtheorem{prop}{Proposition}[section]
\newtheorem*{prop*}{Proposition}
\newtheorem{lemma}[prop]{Lemma}
\newtheorem{sublemma}{Sub-lemma}[prop]

\newtheorem{rk}[prop]{Remark}

\newtheorem{coro}[prop]{Corollary}
\newtheorem{ques}{Question}

\newcommand\numberthis{\addtocounter{equation}{1}\tag{\theequation}}

\newcommand{\Symp}{\mathrm{Symp}}
\newcommand{\Ham}{\mathrm{Ham}}
\newcommand{\Homeo}{\mathrm{Homeo}}

\newcommand{\R}{\mathbb{R}}
\newcommand{\N}{\mathbb{N}}

\newcommand{\Z}{\mathbb{Z}}

\newcommand{\Diff}{\mathrm{Diff}}

\newcommand{\Auto}{\mathrm{Auto}}
\begin{document}
	\begin{abstract}
		We present a way of constructing non-autonomous Hamiltonian diffeomorphisms with roots of all orders by adapting the Anosov-Katok construction. This answers a question by Kathryn Mann and Egor Shelukin. Additionally, we construct an action of the rationals by diffeomorphism on any manifold that is not $ C^0$-continuous with respect to the Euclidean topology on $\mathbb Q$.
    \end{abstract}
	\maketitle

\section{Introduction}\label{sec: intro}

\subsection{Main result and motivation}\label{subsec: Main result}
Let $(M,\omega)$ be a symplectic manifold. We say that a Hamiltonian diffeomorphism is \textit{autonomous} when it is the time-1 map of a flow generated by a time-independent Hamiltonian. The set of all autonomous Hamiltonian diffeomorphisms is denoted by $\Auto(M,\omega)$.  Autonomous Hamiltonian diffeomorphisms enjoy properties that are generally not true for general Hamiltonian diffeomorphisms; for example, one can show that if $H$ is a time-dependent Hamiltonian, then each level set $H^{-1}(\{c\})$ for $c \in \R$ is preserved under the flow $\varphi^t_H$.

Let $(M,\omega)$ be a symplectic manifold. We say that a Hamiltonian diffeomorphism is \textit{autonomous} when it is the time-1 map of a flow generated by a time-independent Hamiltonian. The set of all autonomous Hamiltonian diffeomorphisms is denoted by $\Auto(M,\omega)$.  
In many ways, the set $\Auto(M,\omega)$ is believed to be small in the group of all Hamiltonian diffeomorphisms $\Ham(M,\omega)$, as first pointed out by Polterovich and Shelukhin \cite{PS}. This is a symplectic take on Palis' heuristic \cite{Pal}: ``Vector fields generate few diffeomorphisms". Polterovich and Shelukhin \cite{PS}, followed by Polterovich, Shelukhin and Stojisavljevi\'c \cite{PSS}, Zhang \cite{Zha}, Chor \cite{Cho20}, and Khanevsky \cite{Kha22} indeed showed that for a wide class of closed symplectic manifolds, there exist Hamiltonian diffeomorphisms arbitrarily far from elements of $\Auto(M,\omega)$ with respect to the Hofer norm. A key observation for the proof of the result of Polterovich and Shelukhin \cite{PS} is that all autonomous Hamiltonian diffeomorphisms admit roots of all orders. This comes from the fact that if the time-1 map of the Hamiltonian flow generated by $H \colon M \to \R$ is $\varphi$, then the time-1 map of the Hamiltonian flow generated by $H/k$ is a $k$-th root of $\varphi$. It is thus a natural question to ask whether the only maps that have roots of all orders are autonomous. In  \cite[Section 3]{She}, Shelukhin mentions that there are no known examples of non-autonomous Hamiltonian diffeomorphisms with roots of all orders. The question about their existence has been known in the community since around 2015 \cite{Sto} and was first asked by Kathryn Mann. In this paper, we show the existence of non-autonomous Hamiltonian diffeomorphisms of general symplectic manifolds that admit roots of all orders. Our main result is the following.

\begin{thm}\label{thm: Q to Ham}
For any symplectic manifold $(M,\omega)$ there exists an injective group homomorphism $$\Phi:\mathbb Q \to \Ham(M,\omega)$$ such that $\Phi(\mathbb Q)\cap \Auto(M,\omega) = \{id\}.$
\end{thm}

Since $\Phi(1) \in \Ham(M,\omega)$ is not autonomous but has roots of all orders, the following corollary is immediate.

\begin{coro}\label{coro: answer main question}
For any symplectic manifold $(M,\omega)$, there exists a Hamiltonian diffeomorphism $\varphi$ that has roots of all orders but is not autonomous.\qed
\end{coro}

We will see that the image of the constructed group homomorphisms $\Phi$ lies in the $C^\infty$-closure of the set of autonomous diffeomorphisms. Naturally, we ask the following question.

\begin{ques}
Does there exist a Hamiltonian diffeomorphism that has roots of all orders but does not lie in the $C^\infty$-closure of the set of autonomous diffeomorphisms?
\end{ques}

As we will show shortly, the proof of Theorem \ref{thm: Q to Ham} is a consequence of the following proposition.

\begin{prop}\label{prop: non-aut with roots}
For any symplectic manifold $(M,\omega)$ there exists a group homomorphism
\[\Phi\colon \mathbb Q \to  \Ham(M,\omega)\]
and an open set $\mathcal{U}\subset M$ such that for all $\ell \in \Z_{\geq 1}$ there exists a point $x \in M$ with a positive orbit
\[\mathcal{O}(\Phi(\ell),x):=\{x, \Phi(\ell)(x), \Phi(2\ell)(x),\ldots\}\]
that is dense in $\mathcal{U}$.
\end{prop}

This proposition implies Theorem 1, as we show now. The existence of orbits of $\Phi(\ell)$ that are dense in an open set implies that they are not autonomous. 

\begin{proof}[Proof of Theorem \ref{thm: Q to Ham}] Suppose that there exists an autonomous Hamiltonian diffeomorphism $\varphi\in \Auto(M,\omega)$, generated by $H$ such that there exists a point $x$ of $M$ with an orbit $\mathcal{O}(\varphi,x)$ that is dense in an open set $\mathcal{U}\subset M$. Since autonomous Hamiltonians are preserved by the action of their flow, we have that $H$ is constant on $\mathcal O(\varphi,x)$. However, since the orbit is dense in $U$, the Hamiltonian $H|_U$ restricted to $U$ must be constant. Consequently, $\varphi$ acts as the identity on $U$, and the orbit $\mathcal O(\varphi,x)$ consists of only one point. This contradicts the density of the orbit in the open set $U$.

After applying this argument to $\Phi(\ell)$ for any $\ell\in \mathbb Z_{\geq 1}$ we can conclude that \[\Phi(\mathbb Q)\cap \Auto(M,\omega)=\{Id\}\] and therefore finish the proof of Theorem \ref{thm: Q to Ham} assuming Proposition \ref{prop: non-aut with roots}.
\end{proof}
As we shall see, the arguments we will use are not specific for the symplectic setting but can be applied to any manifold. Naturally, we present some more general applications of our argument to the set of diffeomorphisms of a compact manifold. To study this group, we will equip it with the $C^0$-topology, which is induced by the $C^0$-metric. This metric is defined on the space of $\Homeo(M)$ by first choosing an auxiliary Riemannian metric $g$ and then defining $$d_{C^0}(f,g) = \max_{x\in M}d_g(f(x),g(x)).$$
Similar to before, we show the existence of exotic group homomorphisms from the set of rationals to diffeomorphism groups of manifolds.
\begin{thm}\label{thm: non-continuous}
      For any compact manifold $M$, there exists a group homomorphism $$\Phi:\mathbb Q \to \Diff_c^\infty(M)$$ that is not $C^0$-continuous. 
\end{thm} 
The following two propositions imply this theorem. Firstly, for manifolds of dimension at least $2$, we can make a stronger statement as above. \begin{prop}\label{prop:exoticrationals}
    For any manifold $M$ of dimension at least $2$, there exist injective non-continuous group homomorphisms $$\Phi:\mathbb Q \to \Diff^\infty_c(M)$$ such that the image of $\Phi$ has no isolated points.
\end{prop} 
We will prove this result by applying the Anasov-Katok construction similarly to how we will do for the proof of Proposition \ref{prop: non-aut with roots}. However, the argument in the proof fails for dimension one due to the lack of $k$-transitive diffeomorphisms (that is, diffeomorphisms that send a tuple of $k$ isolated points to an arbitrary tuple of $k$ isolated points).

\begin{rk}
For manifolds $M$ that admit a smooth free $S^1$-action, a group homomorphism $S:S^1\to \Diff(M)$ that is continuous with respect to the $C^0$-topology, one can construct non-continuous group homomorphisms $\mathbb Q\to \Diff(M)$ more directly using that the identity has many $p$\textsuperscript{th} roots for any prime $p$. This approach has been inspired by discussion with Ibrahim Trifa.
\end{rk}

\begin{rk}
    The group homomorphism of Theorem \ref{thm: Q to Ham} can be assumed to be non-continuous. On the other hand, the map $\Phi(1)\in \Diff(M)$ from Theorem \ref{thm: non-continuous} can be assumed to be a non-autonomous Hamiltonian diffeomorphism. 
\end{rk}

\subsection{Organization of the paper}

The paper is organised as follows. In Section \ref{sec: proof of prop non-aut with roots}, we recall the Anosov-Katok construction and we prove Proposition \ref{prop: non-aut with roots} by adapting the construction to our setting. In Section \ref{sec: proof of non continuity}, we expand the proof of Proposition \ref{prop: non-aut with roots} in order to prove Theorem \ref{thm: non-continuous}. In Section \ref{sec: gamma topology}, we study the continuity of the map defined in Theorem \ref{thm: Q to Ham} for two other natural topologies on the group of Hamiltonian diffeomorphisms, namely the $\gamma$- and Hofer-topologies.

\subsection{Acknowledgments} 
The authors of this paper would like to thank P.~Feller, F.~Le~Roux, L.~Polterovich, V.~Stojisavljević and I.~Trifa for their interest and comments about this work. The authors would like to thank especially E.~Çinelli and S.~Seyfaddini for comments and insightful discussion throughout the project. The second author also would like to thank M.~Burger for a discussion that led him to better understand the mathematical context of this article, as well as A.~Ulliana and H.~Wu for organising a seminar that contributed to the birth of this article. During this project, B.S. was supported by ERC Starting Grant 851701.

\section{Proof of Proposition \ref{prop: non-aut with roots}}\label{sec: proof of prop non-aut with roots}

This section is dedicated to the proof of Proposition \ref{prop: non-aut with roots}. \emph{From now on, we will always assume that $M$ is a compact manifold, unless otherwise specified.} We recall first some basic definitions and notations of symplectic geometry. We then describe our setup for the Anosov-Katok construction that differs from the usual one. We finish the proof of Proposition \ref{prop: non-aut with roots} in Section \ref{subsec: Scheme} and \ref{subsec: Achieving dense orbits}.

\subsection{Preliminaries of symplectic geometry}
A symplectic manifold $M$ is a manifold endowed with a symplectic form $\omega$, i.e. a closed non-degenerate 2-form. For any symplectic manifold $(M,\omega)$, we denote by $\Symp(M,\omega)$ the group of \textit{symplectomorphisms} of $M$. Let $H_t \colon M \to \R$ be a time-dependent compactly supported Hamiltonian; one defines the \textit{Hamiltonian vector field $X_H$ associated with} $H$ to be the unique vector field that verifies the following formula
\[ \iota_{X_H}\omega=-\mathrm{d}H.\]
The flow of the vector field $X_H$ is called the \textit{Hamiltonian flow} of $H$ and is denoted $\varphi^t_H$. The time-1 maps of Hamiltonian flows are called \textit{Hamiltonian diffeomorphisms}, and altogether, they form a subgroup of $\Symp(M,\omega)$ denoted by $\Ham(M,\omega)$.

\subsection{Setup of the Anosov-Katok construction}
The proof consists in constructing a Hamiltonian diffeomorphism with the Anosov-Katok construction, first introduced in \cite{AK70}. Their construction uses a Hamiltonian $S^1$-action on a symplectic manifold. We generalise the construction to perform it on symplectic manifolds $(M,\omega)$ with a smooth Hamiltonian action $S:\mathbb R \to (\Ham(M,\omega),\tau_{C^\infty})$ such that there exists an invariant connected open subset $U\subset M$ such that the closure $\overline U$ is compact and upon which the restriction of the Hamiltonian action $S(\cdot)|_U$ is 1-periodic and induces a free circle action without fixed points denoted by
\[S|_U:S^1\to \Symp(U,\omega).\]
Here, $\tau_{C^\infty}$ denotes the Whitney topology (since we assume that $M$ is compact, the weak and strong Whitney topologies agree, and we can say Whitney topology). This topology admits a metric that turns $(\Diff(M),\tau_{C^\infty})$ into a complete metric space.
From the data of $S$, we will construct a sequence of group homomorphisms $(\Phi_n:\mathbb R\to \Ham(M,\omega))_n$ such that the following hold. \begin{itemize}
    \item[(i)] The group homomorphisms $\Phi_n$ are autonomous flows.
    \item[(ii)] For any $ q\in \mathbb Q$ the sequence $(\Phi_n(q))_n$ converges in the $C^\infty$-topology. The limits of all rational points fit together to form a group homomorphism 
    $$\Phi:\mathbb Q\to \Ham(M,\omega).$$
    \item[(iii)] For every $\ell\in \mathbb Z_{\geq 1}$, the diffeomorphisms $\Phi(\ell)$ have a dense orbit in $U$.
\end{itemize}
The existence of such a group homomorphism $\Phi:\mathbb Q\to \Ham(M,\omega)$ proves Proposition \ref{prop: non-aut with roots}. In the subsequent two sections, we first show how the sequence $(\Phi_n)_n$ that satisfies (i) and (ii) is constructed and then argue that we can improve the construction, additionally achieving (iii).

Let us assume for now that the above construction can be carried out, then once a suitable group action $S\colon \R \to \Ham(M,\omega)$ exists, we have proven Proposition \ref{prop: non-aut with roots}. We now show how to construct a suitable action $S$ on every symplectic manifold.

Our construction will take place in $\R^{2n}$ endowed with the standard symplectic form $\omega_0$. Darboux's theorem allows us to embed this construction in any symplectic manifold. Let $u_0:\mathbb R_{\geq 0} \to [0,1]$ be a function given by 
$$u_0(x)= \begin{cases}
    1-x, & \text{for }x\in [0,1]\\
    0,& \text{otherwise. }
\end{cases}$$
We now locally smoothen $u_0$ in $[0,\varepsilon)$ and $(1-\varepsilon, 1+\varepsilon)$, two small neighbourhoods of $0$ and $1$ respectively. By doing this, we can obtain a smooth compactly supported function $u:\mathbb R_{\geq0}\to[0,1]$ that is locally constant at $0$ and coincides with $u_0$ on $(\varepsilon, 1-\varepsilon)$. We now define the radial Hamiltonian on $\mathbb R^{2n}$ by 
$$G(x) = u(|x|).$$
The flow generated by $G$ is a suitable Hamiltonian action $S:\mathbb R \to \Ham(\mathbb  R^{2n},\omega_0)$. That is, denoting $B_r$ for the open ball of radius $r$ centred at the origin, the open set $U = B_{1-\varepsilon}-\overline{B_{\varepsilon}}$ is invariant under $S$ and $S$ acts on $U$ as a free circle action.

\subsection{Scheme of the construction}\label{subsec: Scheme}
We construct the sequence of autonomous morphisms $\Phi_n$ with the Anosov-Katok method, following the scheme described in \cite{LRS22}. Our construction differs from theirs in that we use a diagonal argument to impose the convergence of a countable number of maps simultaneously. We denote by $d_{C^\infty}$ a distance on $\Symp(\text{supp}(S),\omega)$ that induces the Whitney $C^\infty$-topology and turns it into a complete metric space.

The backbone of this construction is the Hamiltonian action $S$ which restricts to a circle action on an invariant set $U$. Then the morphisms $\Phi_n$ are defined as conjugations of the action $S$ stretched by the factor $\alpha_n$. In other words, we have
$$\Phi_n(\lambda) = H_nS_{\lambda \alpha_{n+1}}H_n^{-1},$$
where $H_n \in \Ham(U,\omega)$ is a Hamiltonian diffeomorphism supported in $U$ and $\alpha_n\in \mathbb Q$. In addition to this, we also demand that when written in irreducible form $\alpha_n=\frac{p_n}{q_n}$ is such that $p_n$ is divisible by $n!$. In what follows, we will need fast convergence of the sequence of rational numbers $\alpha_n$, we describe briefly why the constraint on the numerator is harmless with respect to that.
\begin{lemma}\label{lemma: density with constraint on the numerator}
Let $n$ be any integer, then the set of fractions that can be written with a numerator that is divisible by $n!$ in their irreducible representation is dense in $\R$.
\end{lemma}

\begin{proof}
Let $\varepsilon>0$ be a positive number, pick $q$ a prime number larger than $n!$ and $\frac{2n!}{\varepsilon}$. Then, the set of all irreducible fractions with denominator $q$ and numerator divisible by $n!$ forms is an $\varepsilon$-net inside $\mathbb R$. Since we can do this for all $\varepsilon>0$, the result follows.
\end{proof}

The construction will be inductive; we construct the Hamiltonian diffeomorphism $H_n$ such that $H_n:=H_{n-1} \circ h_n$ for some Hamiltonian diffeomorphism $h_n \in \Ham(U,\omega)$ that commutes with $S_{1/q_n}$. In this way for every $\lambda \in \mathbb Q$ with denominator in $\llbracket 1, n \rrbracket = [1,n]\cap \mathbb Z$, the following holds
\begin{align}\label{equ.commutingcondition}
\Phi_{n-1}(\lambda) = H_{n-1} S_{\lambda \alpha_{n}}H_{n-1}^{-1} = H_{n} S_{\lambda \alpha_{n}}H_{n}^{-1}.
\end{align}
The identity (\ref{equ.commutingcondition}) implies that if we choose $\alpha_{n+1}$ close enough to $\alpha_n$, we can conclude that for all $k \in \llbracket 1, n \rrbracket$, we have
\begin{align*}
    d_{C^\infty}\left(\Phi_n\left(1/k\right),\Phi_{n-1}\left(1/k\right)\right)&= d_{C^\infty}(H_n S_{\alpha_{n+1}/k}H_n^{-1}, H_{n-1} S_{\alpha_{n}/k}H_{n-1}^{-1})\\
    & = d_{C^\infty}(H_n S_{\alpha_{n+1}/k}H_n^{-1}, H_{n} S_{\alpha_{n}/k}H_{n}^{-1})\\ &\leq \dfrac{1}{2^n}. \numberthis \label{equ.convergenceofmaps}
\end{align*}
Note that from the condition imposed on the numerator of $\alpha_n$ it follows that for all $\lambda \in \mathbb Q$ the equation (\ref{equ.convergenceofmaps}) will hold eventually for large enough $n$. This means that for every $\ell\in \mathbb Z_{\geq 1}$, the sequence of maps $(\Phi_n(\frac{1}{\ell}))_n$ is a Cauchy sequence and thus converges in $C^\infty$-topology to a diffeomorphism that we denote $\Phi(\frac{1}{\ell})$. Given the pointwise convergence of the sequence $(\Phi_n)_n$ on $\mathbb{Q}$ and the fact that each $\Phi_n$ is a group homomorphism, it follows directly that the sequence converges to a map $\Phi \colon \mathbb{Q} \to \Diff(M, \omega)$, which is also a group homomorphism.
Since the $C^k$-flux conjecture holds for $1\leq k\leq \infty$ \cite{On06}, we can conclude that the limit diffeomorphisms lie in $\Ham(M,\omega)$, that is, we have constructed
$$\Phi: \mathbb Q \to \Ham(M,\omega).$$

This construction is very general; now one can demand additional properties for the sequences $(\alpha_n)_n$ and $(H_n)_n$ to achieve specific properties in the limit $\Phi$. For example, as described in \cite{LRS22}, by demanding light conditions for $H_n$, the limit diffeomorphisms can be forced to have a dense orbit in $U$ and therefore are not autonomous.

Although the condition on $h_n$ might seem very strong, it can be reformulated as follows. For any $n\in \mathbb Z_{\geq 1}$, we can consider the quotient $(U/S_{1/q_n},\sigma)$ which is a smooth symplectic manifold with a symplectic structure $\sigma$ induced by $\omega$. Note that there is a smooth covering
$$\pi:U\to U/S_{1/q_n}.$$
Any Hamiltonian diffeomorphism $g_n\in \Ham(U/S_{1/q_n},\sigma)$ admits a lift $h_n\in \Ham(U,\omega)$ that commutes with $S_{1/q_n}$ and therefore (\ref{equ.commutingcondition}) is achieved.

\subsection{Achieving dense orbits}\label{subsec: Achieving dense orbits}
The proof that a dense orbit in the limit diffeomorphisms of the Anosov-Katok construction can be achieved as described in \cite{LRS22}. For the convenience of the reader and since we want the limit diffeomorphisms $\Phi(\ell)$ to have dense orbits for all $\ell\in \mathbb Z_{\geq 1}$, we describe how this construction can be achieved.

 Fix an auxiliary Riemannian metric on $M$ and denote the induced distance by $d$. A subset $F\subset U$ is called \emph{$\varepsilon$-dense} in $U$ if
$$\sup_{x\in U}\inf_{y\in F}d(x,y)<\varepsilon.$$
Now choose a fixed sequence $(\varepsilon_n)_n$ of real numbers converging to $0$. We will choose $g_{n}\in \Ham(U/S_{1/q_n},\sigma)$ so that $\Phi_n(n!)$ has an orbit $\mathcal{O}(\Phi_n(n!),x_n)$ that is $\varepsilon_n$-dense in $U$. Note that since for $\ell = 1,...,n$ we have
\[\mathcal{O}(\Phi_n(n!),x_n) \subset \mathcal{O}(\Phi_n(\ell),x_n),\]
and thus $\Phi_n(\ell)$ has an $\varepsilon_n$-dense orbit in $U$. Fix $\ell \in \mathbb Z_{\geq 1}$, for $n$ an integer large enough, the map $\Phi_n(\ell)$ has an $\varepsilon_n$-dense orbit, we obtain as in \cite[Section 2.3]{LRS22} that fast convergence implies that $\Phi(\ell)$ has an $\varepsilon$-dense orbit in $U$ for any $\varepsilon>0$. Here we use the compactness closure of $U$ to conclude that such orbits can be assumed to be finite. Again, as in \cite{LRS22}, this implies that $\Phi(\ell)$ has a dense orbit in $U$. This would imply Proposition \ref{prop: non-aut with roots}.

We still need to show that the sequence $(g_n)_n$ can be chosen such that $\Phi_n(n!)$ has an $\varepsilon_n$-dense orbit. We assume that we already constructed $\alpha_n$ and $H_{n-1}$ such that the map $\Phi_{n-1}((n-1)!)=H_{n-1} S_{\alpha_n(n-1)!}H_{n-1}^{-1}$ has an $\varepsilon_{n-1}$-dense orbit and construct $H_n$ and $\alpha_{n+1}$ such that $\Phi_n(n!)$ has an $\varepsilon_n$-dense orbit. Set
$$\pi_n:U\to U':= U/S_{1/q_n}$$ 
to be the standard projection. Choose a finite set of points $F_1\subset U'$ such that $\pi_n^{-1}(F_1)$ is $\eta_{n}$-dense in $U$ where $\eta_{n}$ is chosen such that $H_{n-1}$ maps any ball of diameter $\eta_{n}$ inside a ball of diameter $\varepsilon_{n}$. Note that this is possible since $U$ lies in the support of the flow $S$, which is compact. We choose a set $F_1\subset U'$ with the same cardinality as $F_1$, all lying in the projection $\pi(C)$ of a single orbit $C$ of the circle action on $U$. Since the group of Hamiltonian diffeomorphisms is $|F_1|$-transitive (Lemma \ref{lemma: n-trans of Ham}), we can pick $g_n\in \Ham(U')$ such that $g_n(F_1)=F_2$.

\begin{lemma}\label{lemma: n-trans of Ham}
    For any manifold $M$ and any $p$-tuples $(x_1,...,x_p)$ and $(y_1,...,y_p)$ where $x_i \neq x_j$ and $y_i\neq y_j$ for $i\neq j$ there exists a compactly supported diffeomorphism $\psi$ such that $\psi(x_i) = y_i$ for all $i$. If we assume the existence of a symplectic form on $M$, then $\psi$ can be assumed to be a Hamiltonian diffeomorphism. \qed
\end{lemma}

We now define $h_n$ as the lift of $g_n$ to $U$. The image $h_n(C)$ of the orbit $C$, is $\eta_n$-dense in $U$ and thus, $H_{n-1}(h_n(C))$ is $\varepsilon_n$-dense in $U$. We choose a rational number $\alpha_{n+1}=\frac{p_{n+1}}{q_{n+1}}$ with $p_{n+1}$ divisible by $(n+1)!$ such that there is a discrete orbit $C' \subset C$ under $S_{\alpha_{n+1}n!}$ that contains a subset close enough to $\pi^{-1}(F_1)$ such that $H_{n-1}(h_n(C'))$ is $\varepsilon_n$-dense in $U$, this is possible if $\alpha_{n+1}$ is close enough to $\alpha_n$. We define $H_n:=H_{n-1}\circ h_n$ and $\Phi_n=H_nS_{\alpha_{n+1}}H_n^{-1}$. The set $H_n(C')$ is then an $\varepsilon_n$-dense orbit of $\Phi_n(n!)$. This concludes the proof of Proposition~\ref{prop: non-aut with roots}.

\color{black}%%%%%%%%%%%%%%%%%%%%%%%%%%%%%%%%%%%%%%%%%
\section{Proof of Theorem \ref{thm: non-continuous}}\label{sec: proof of non continuity}

As explained in the introduction, the proof of Theorem \ref{thm: non-continuous} follows from Proposition \ref{prop:exoticrationals}. This section is devoted to the proof of Proposition \ref{prop:exoticrationals}. The construction scheme is the same as in the previous section. In fact, the construction of $\Psi$ can be performed analogously without knowledge of a symplectic form, where diffeomorphisms replace Hamiltonian diffeomorphisms.

\subsection{Constructing exotic rationals in $\Diff(M)$}
Starting from a continuous group homomorphism $\mathbb R\to (\Diff(M),\tau_{C^\infty})$ we will build an exotic group homomorphism $\Phi:\mathbb Q\to \Diff(M)$ similarly to the previous section. However, we now demand the group homomorphism $\Phi$ to fulfil further conditions. 

\begin{prop}\label{prop: homomorphism to diff}
For all smooth manifold $M$ of dimension greater or equal than 2, there exists a group homomorphism $$\Phi:\mathbb Q\to \Diff(M)$$ satisfying the following properties:
 \begin{enumerate}[label=(\roman*)]
        \item There exists a real number $c>0$ such that for any prime $p$ large enough we have $d_{C^0}(id,\Phi(1/p))>c$.
        \item There exists a diverging sequence of natural number $(p_n)_n$ such that $\Phi(1/p_n)$ converges to the identity.
        \item For any $\ell$ there exists an $x\in M$ such that $\{\Phi(\ell n)x\}_{n\in \mathbb Z_{\geq1}}$ is dense in some open set $U\subset M$.
    \end{enumerate}
\end{prop}

Note that $(iii)$ automatically implies that $\Phi$ is injective. Otherwise, there would be a natural number $\ell$ such that $\Phi(\ell)= id$, and therefore there cannot be a dense orbit of $\Phi(\ell)$ in any open set. To obtain $\Phi$, we again apply the Anosov-Katok method to a group homomorphism $S:\mathbb R\to \Diff(M)$ but demand the sequences $(\alpha_n)_{n}$ and $(h_n)_n$ to satisfy additional conditions. Let us first assume that we have a continuous group homomorphism $S:\mathbb R \to \Diff(M)$ and an open set $U\subset M$ where $S$ is invariant and so that $S$ restricts $U$ to a group action $S|_U:\mathbb R \to \Diff(U)$ with a discrete kernel and therefore induces a free circle action by diffeomorphisms, that is, a continuous group homomorphism $S|_U:S^1\to \Diff(U)$ such that $S|_U(\xi)(x) =x$ for some $x\in U$ implies that $\xi =1$. We again choose a sequence $(h_n)_n$ of diffeomorphisms of $M$ compactly supported in $U$ and a sequence of rational numbers $(\alpha_n)_n$ and define 
$$\Phi_n:\mathbb Q\to \Diff(M),\quad \lambda \mapsto  H_n \circ S_{\lambda\alpha_{n+1}} \circ H_n^{-1}$$ 
where $H_n = h_n\circ...\circ h_1$. Like in the previous section, the sequence $(\Phi_n)_n$ will converge to a group homomorphism $\Phi$. For $\Phi$ to satisfy condition $(i)$ from Proposition \ref{prop: homomorphism to diff} we choose $(h_n)_n$ and $(\alpha_n)_n$ in such a way that there exists a real number $c>0$ such that 
\begin{equation}\label{equ: non-continuity}
d_{C^0} (Id,\Phi_{p-1}(1/p))> 2c.
\end{equation}
The sequence $(\Phi_n)_n$ of group homomorphism will again fulfill condition (\ref{equ.convergenceofmaps}) and, therefore we have that $$d_{C^0}(Id,\Phi(1/p)) \geq d_{C^0}(Id,\Phi_{p-1}(1/p))-d_{C^0}(\Phi_{p-1}(1/p),\Phi(1/p))> 2c - \sum_{k \geq p}2^{-k}.$$
Thus, for large enough primes $p$ we have that 
$$d_{C^0}(Id,\Phi(1/p)) > c.$$
To achieve $(ii)$ it is enough to refine the choice of $\alpha_n = {p_n\over q_n}$ such that $p_n|p_{n+1}$ and $q_n$ is chosen to be large enough. Choosing $(\alpha_n)_n$ correctly will imply that $\Phi_n(1/p_n)$ is bounded by an arbitrarily small number.
As will become apparent later, all this can be done simultaneously while also ensuring $\Phi(\ell)$ has a dense orbit in an open set such that $\Phi$ also satisfies $(iii)$.\\
Let us first concentrate on $(i)$. For this, we choose $(h_n)$ in a particular way where $h_n$ is supported on an open set $U_n$ where $$U_1\subset U_2 \subset U_3 \subset ... \subset U$$ is an increasing sequence of open sets. This gives us control of the behaviour of $H_n = H_{n-1}\circ h_n$ on $U_n-U_{n-1}$ since $H_{n-1}$ is supported on $U_{n-1}$.
More concretely, to prove (\ref{equ: non-continuity}), we will use the following lemma.
\begin{lemma}\label{lemma: sequenceofopensets}
For any connected open set $U\subset M$ and free circle action $S$ on $U$ there exists a sequence of connected open sets $$U_1\subset U_2 \subset U_3 \subset ... \subset U$$ and a positive number $c>0$ such that the following is fulfilled. \begin{enumerate}[label=(\roman*)]
    \item $\bigcup_n U_n$ is dense in $U$.
    \item For all $n$ large enough, there exists $x_{n+1},y_{n+1}$ in the interior of $ U_{n+1}-U_n$ such that $d(x_{n+1},y_{n+1})>c$.
    \item The sets $U_n$ are invariant under the circle action $S$.
\end{enumerate}
\end{lemma}
\begin{proof}
Let us first consider the case when dim$(M) = 2$. Then, since we assumed the circle action to be free, the open set $U$ is diffeomorphic to a cylinder with two boundary components $A,B$. Consider $$U_n = \left\{x\in U| \inf_{t\in S^1}d(S_tx,\partial U)> {1\over n}\right\}$$ which is an exhausting sequence of $U$ by open connected subsets which are invariant under the circle action. Now let $d(A,B)= 2c>0$ and note that for $n$ large enough, $U_{n+1}-U_n$ has two connected components, each close to a boundary component $A$ or $B$. Picking any $x_{n+1},y_{n+1}$ in distinct connected components of $U_{n+1}-U_n$ will ensure that (ii) holds. 

If the dimension of $M$ is $n>2$, we choose $x_0, y_0 \in U$ with distinct continuous orbits $C_{x_0},C_{y_0}$ under $S$. Let $X$ denote the vector field that induces the circle action. That is $$X = {d\over dt}|_{t=0}S_t.$$ For any $x\in U$ and $\varepsilon>0$ choose a smoothly embedded $(n-1)$-dimensional open disk $D_{x,\varepsilon}$ centred at $x$ that is transverse to $X$ and contained in a $\varepsilon$-ball centred at $x$. For $\varepsilon$ small enough, note that for any $z\in D_{x,\varepsilon}$ we have $\{S_tz\}_{t\in S^1}\cap D_{z,\varepsilon} = \{z\}$. In this case, the set $N_{x,\varepsilon} = \bigcup_{t\in S^1}S_tD_{x,\varepsilon}$ is an $S^1$-invariant neighbourhood of the orbit of $x$. Moreover, for any $\delta>0$ there exists a $\varepsilon>0$ such that $N_{x,\varepsilon}$ is contained in a $\delta$-neighbourhood of the orbit of $x$. Moreover, for $\varepsilon$ small enough, note that $N_{x_0,\varepsilon}$ and $N_{y_0,\varepsilon}$ are disjoint tubular neighbourhoods of $C_{x_0}$ and $C_{y_0}$. Now consider $$U_n = U-(N_{x_0,\varepsilon/n}\cup N_{y_0,\varepsilon/n})$$ and note that these sets form an exhausting sequence $U-(C_{x_0}\cup C_{y_0})$ of open connected sets invariant under the circle action. Since $U-(C_{x_0}\cup C_{y_0})$ is dense in $U$ the sequence fulfils (i). For similar reasons as in the case of dimension 2
they also fulfil (ii) if we define $d(C_{x_0},C_{y_0}) = 2c >0$. \end{proof} 

\color{black}%%%%%%%%%%%%%%%%%%%%%%%%%%%%%%%%%%%%%%%%%%%%%%%%%%%%%%%%%%

We now show how to construct $(h_n)_n$ so that $(\Phi_n)_n$ satisfies (\ref{equ: non-continuity}). Namely, we first proceed in the same way as in Section \ref{subsec: Scheme}; however, we require that $h_n$ is supported in the open set $U_n$. Recall that we achieve this by choosing diffeomorphisms of the quotients $U_n /S_{1 / q_n}$ and lifting them to compactly supported diffeomorphisms of $U$ that send a discrete finite orbit $C$ to an $\eta_n$-dense set. Without loss of generality, we may assume that the points $x_n$ and $y_n$ given by Lemma \ref{lemma: sequenceofopensets} are not contained in $C$ or its image. In addition to what we demanded $h_n$ to fulfil up to now, we further ask that when $n+1$ is prime, $x_n$ be sent to $x_n$ and $S_{\alpha_{n}/n+1}x_n$ be sent to $y_n$. In general, this is not possible since $h_n$ is a lift of a diffeomorphism of the quotient $U_n /S_{1 / q_n}$ where we might identify $x_n$ with $S_{\alpha_n/n+1}$. This is the case if and only if $\alpha_n = {p_n\over q_n}$ is such that $n+1$ divides $p_n$. However, if $n+1$ is prime, we will show that this is not an issue in a moment. As we will see, the described choice of $(h_n)_n$ together with fast convergence of $(\alpha_n)_n$ implies that $\Phi$ satisfies conditions $(i)$ and $(iii)$ from Proposition \ref{prop: homomorphism to diff}.
Given the goal of also achieving $(ii)$ simultaneously, we will ask that $p_n$ and $q_n$ fulfil even further conditions. Namely, what we need for $(ii)$ is that $p_n$ divides $p_{n+1}$. We show that this and all of the above conditions can be achieved by a set of rational numbers that are dense in $\mathbb Q$, ensuring that the sequence $(\alpha_n)_n$ can converge arbitrarily fast. Let $\mathcal Z_n$ be the set of natural numbers that are given by a product of numbers $1,...,n$ 
$$\mathcal Z_n = \{2^{\ell_1}\cdot \ldots \cdot n^{\ell_{n-1}}|\ell_i \in \mathbb Z_{\geq 0}\}.$$
Note that by demanding that $p_n\in \mathcal Z_n$ we automatically have that $n+1$ does not divide $p_n$ if $n+1$ is prime.
\begin{lemma}
    Let $n\in \mathbb Z_{\geq 2}$ and $p_{n-1}\in \mathcal Z_{n-1}$. Then the set of rational numbers ${p_n\over q_n}$ in irreducible form, where $p_n$ belongs to $\mathcal{Z}_n$ is divisible by $n!$ and $p_{n-1}$, are dense in $\mathbb Q$. 
\end{lemma}
\begin{proof}
Let ${p\over q}$ be a rational number and let $\varepsilon>0$. For any choice $p_n$ we first set $\tilde q_n = \lfloor{p_nq\over p}\rfloor$, at the end of the proof we will let $p_n$ and thus $\tilde{q}_n$ run to infinity. We now use the following Sub-lemma.

\begin{sublemma}\label{sublemma: NT sublemma}
Let $r_1, \ldots, r_k$ be natural numbers. Then, there exists some natural number $P\coloneqq P(r_1, \ldots, r_k)$ such that for all $x \in \N$ there is a natural number in the interval $\llbracket x-P, x+P \rrbracket$ that is prime with all $r_i$'s.
\end{sublemma}

\begin{proof}
Let $P=r_1 \cdot\ldots \cdot r_k$, then in each interval of the form $\llbracket x-P, x+P \rrbracket$ there is at least one number $y$ that verifies that $y \equiv 1 \pmod{P}$. This implies that this number is not divisible by any of the $r_i$'s.
\end{proof}

From the Sub-lemma, we define $P \coloneqq P(n,p_{n-1})$ thus there exists $q_n\in \llbracket \tilde q_n-P,\tilde q_n +P\rrbracket$ that is not divisible by any prime smaller than $n$ or $p_{n-1}$. Thus, $p_n\over q_n$ in irreducible from is given by $\bar p_n\over \bar q_n$ where $p_{n-1}$ and $n!$ still divide $\bar p_n$. We then have that 
$$\left|{p_n\over q_n}-{p\over q}\right| = {|p_nq-pq_n|\over q_nq} \leq {|p_nq-p\tilde q_n|\over q_n q} + {Pp\over q_nq  } \leq {(P+1)p\over q_n q}.$$
Note that $p_n\in \mathcal Z_n$ satisfying $n!|p_n$ and $p_{n-1}|p_n$ can be chosen arbitrarily large. By choosing $p_n$ large enough, we achieve ${(P+1)p\over q_n q}< \varepsilon$ and thus 
$$\left|{p_n\over q_n}-{p\over q}\right|<\varepsilon.$$
\end{proof}
Thus, we can assume that the sequence $(\alpha_n)_n$ is chosen such that $\alpha_n = {p_n\over q_n}$ in its irreducible form is such that $n!$ and $p_{\ell}$ divides $p_n$ for any $\ell\leq n$ and additionally, if $n+1$ is prime, $n+1$ does not divide $p_n$.
Moreover, we can choose $\alpha_{n+1}$ arbitrarily close to $\alpha_n$ that satisfies all the required properties discussed in the previous sections and fulfils the following additional constraints. Firstly, we demand $\alpha_{n+1}$ to be so close to $\alpha_n$ such that $h_n\circ S_{\alpha_{n+1}/n+1}(x_n)$ lies in the interior of $U_n-U_{n-1}$.
Now consider the map $\Phi_n(1/ n+1) =  H_{n-1}h_n S_{\alpha_{n+1}/n+1}h_n^{-1} H_{n-1}^{-1}$ and note that 
\begin{align*}
    \Phi_n\left(1 \over n+1\right)(x_n) &= H_{n-1} h_n S_{\alpha_{n+1}/n+1}h_n^{-1} H_{n-1}^{-1}(x_n)
    \\ &= H_{n-1}h_n S_{\alpha_{n+1}/n+1}h_n^{-1}(x_n)\\&=  H_{n-1}(h_n S_{\alpha_{n+1}/n+1}(x_n))
    \\ &= h_n S_{\alpha_{n+1}/n+1}(x_n).
\end{align*}
The second and last equalities come from the fact that $x_n$ and $h_n S_{\alpha_{n+1}/n+1}(x_n)$ lie outside the support of $H_{n-1}$. We choose $\alpha_{n+1}$ possibly even closer to $\alpha_n$ such that $$d(y_n, h_n S_{\alpha_{n+1}/n+1}x_n) < d(x_n,y_n)-2c$$ and therefore ensuring that 
$$d(x_n,\Phi_n(1/ n+1)x_n)>2c.$$
Thus, we can conclude that 
$$d(Id,\Phi_n(1/n+1))>2c$$ 
and therefore $\Phi$ satisfies condition $(i)$. To see that condition $(ii)$ is fulfilled, first note that for any $k\leq n$ we have 
$$\Phi_k(1/p_n) = H_k S_{\ell_{k+1}/ q_{k+1}}H_k^{-1} = H_{k+1}S_{\ell_{k+1}/q_{k+1}}H_{k+1}^{-1} =H_{k+1}S_{\alpha_{k+1}/p_n}H_{k+1}^{-1},$$
where $\ell_{k+1} = p_{k+1}/p_n$. By fast convergence of the sequence $(\alpha_n)_n$ we can assume 
$$d_{C^\infty}(\Phi_k(1/p_n),\Phi_{k+1}(1/p_n)) = d_{C^\infty}(H_{k+1}S_{\alpha_{k+1}/p_n}H_{k+1}^{-1},H_{k+1}S_{\alpha_{k+2}/p_n}H_{k+1}^{-1})\leq 2^{-k-1}.$$
Therefore, we have that 
$$d_{C^\infty}(\Phi_n(1/p_n),\Phi(1/p_n))\leq \sum_{k> n}2^{-k}.$$ 
Since we can assume that $q_n$ is arbitrarily large we may assume that $$d_{C^\infty}(id, \Phi_n(1/p_n)) = d_{C^\infty}(id,H_nS_{1/q_n}H_n^{-1})\leq 2^{-n}.$$ Thus, we have that $$d_{C^\infty}(id,\Phi(1/p_n)) \leq \sum_{k\geq n} 2^{-k}$$ which proof that $\Phi$ satisfies condition $(ii)$. Note that $(iii)$ is fulfilled for the same reasons as in the previous section. 

To complete our argument, we now give an example of an action $S:\mathbb R\to M$ that restricts to a free $S^1$ action on an invariant subset. Consider any smoothly embedded thickened loop $\gamma:S^1\times D\to M$ where $D$ is the $(n-1)$-dimensional open disk. Consider the vector field $X$ on $\gamma(S^1\times D)$ that is given by $$X(t,x) = {d\over dt}\gamma(t,x).$$ Now extend $X$ to a smooth vector field on $M$ and let $S$ be the flow of $X$. On $\gamma(S^1\times D)$ the flow $S$ indices a free $S^1$ action.

\section{Relation with the $\gamma$-topology and Hofer-topology}\label{sec: gamma topology}

In the proof of Theorem \ref{thm: non-continuous}, we saw that the group homomorphism $\Phi \colon \mathbb Q \to \Ham(M,\omega)$ need not be $C^0$-continuous. In this section, we study the continuity properties with respect to other metrics on $\Ham(M,\omega)$. When the map $\Phi$ is continuous for some topology then it extends to a continuous group morphism from the real line to some completion of $\Ham(M,\omega)$. For the $C^0$-topology we proved earlier that this does not always happen. However, the answer to this question is different when we ask for the Hofer- or $\gamma$-topology on $\Ham(\R^{2n},\omega_0)$.
We recall some properties of the $\gamma$-norm and the Hofer norm. The \textit{Hofer norm} is defined in the following way, for all Hamiltonian diffeomorphisms $\varphi\in \Ham(M,\omega),$
\[d_H(Id,\varphi):=\inf\left\{\int_0^1 (\max_M H_t - \min_M H_t) \mathrm{d}t\mid H_t \text{ generates a path } \varphi_t, ~ \varphi_0=Id, ~\varphi_1=\varphi \right\}.\]
As the notation already indicates, the Hofer norm is induced by a metric $d_H$ called the \emph{Hofer-metric}. One can see that the Hofer metric is indeed a bi-invariant metric. The fact that it is non-degenerate is a deep result proven by Lalonde and McDuff in \cite{LM}.

The $\gamma$-norm was first defined by Viterbo in \cite{Vit92} using the theory of generating functions for Lagrangians in the cotangent bundle of a manifold, and it satisfies the following properties.

\begin{prop}[Viterbo]\label{prop: gamma norm}
There exists a map $\gamma \colon \Ham(\R^{2n},\omega_0) \to \R$ that verifies the following properties for all $\varphi, \psi \in \Ham(\R^{2n})$.
\begin{itemize}
    \item[(i)] $\gamma(\varphi \psi)\leq \gamma(\varphi)+\gamma(\psi)$.
    \item[(ii)] $\gamma(\varphi)\geq 0$ and $\gamma(\varphi)=0$ if and only if $\varphi=Id$.
    \item[(iii)] $\gamma(\psi \varphi \psi^{-1})=\gamma(\varphi)$.
    \item[(iv)] $\gamma(\varphi) \leq d_H(\varphi,Id)$.
    \item[(v)] If $H \leq K$ are two Hamiltonians, then $\gamma(\varphi^1_H) \leq \gamma(\varphi^1_K)$.
\end{itemize}
\end{prop}

The normed space $(\Ham(\R^{2n},\omega_0),\gamma)$ is not complete, and we denote $\widehat{\Ham}(\R^{2n},\omega_0)$ its completion called the \textit{Humilière-completion}, see \cite{Hum}. Note that elements of $\widehat{\Ham}(\R^{2n},\omega_0)$ need not be maps. It turns out that $\Phi$ as in Theorem \ref{thm: Q to Ham} constructed above is both $\gamma$- and Hofer-continuous and hence can be extended over the completion; this is the content of the next proposition. Proposition \ref{prop: gamma continuity in general case} and Proposition \ref{prop: the image is closed for gamma} are formulated for the $\gamma$-norm, but they can also be formulated in terms of the Hofer norm. This choice was made because the completion of $\Ham(M,\omega)$ for the Hofer-norm is not studied as much as the Humilière-completion.

\begin{prop}\label{prop: gamma continuity in general case}
The map 
\[\Phi \colon \mathbb{Q} \to \Ham(\R^{2n},\omega_0)\]
as constructed above is $\gamma$-continuous. As a consequence, $\Phi$ extends as a homomorphism in the Humilière-completion
\[\widehat{\Phi} \colon \mathbb R \to \widehat{\Ham}(\R^{2n},\omega_0).\]
Moreover, the map $\widehat{\Phi}$ is injective.
\end{prop}

\begin{proof}
Since $\Phi$ is a group homomorphism, it suffices to show that $\Phi$ is $\gamma$-continuous at 0. 

Let $\Phi_n(t)=H_n S_{\alpha_n t}H_n^{-1}$ be the sequence of maps constructed in Section \ref{sec: proof of prop non-aut with roots} and let $r$ be a rational number. By construction, we know that the sequence of maps $\Phi_n(r)$ converges in the $C^\infty$-topology to $\Phi(r)$, hence also in the $\gamma$-topology. This and the conjugation-invariance of the $\gamma$-norm implies that
\begin{align*}
\gamma(\Phi(r))&=\lim_{n \to \infty} \gamma(\Phi_n(r))\\
&=\lim_{n \to \infty} \gamma(H_nS_{\alpha_n r}H_n^{-1})\\
&=\lim_{n\to \infty} \gamma(S_{\alpha_n r})\\
&=\gamma(S_{\alpha r})
\end{align*}
where $\alpha\in \mathbb R$ is the limit of the sequence $(\alpha_n)_n$.
Therefore, we find $\lim_{r \to \infty}\gamma(\Phi(r))=0$ as expected. Thus, we can define $\widehat{\Phi}\colon \R \to \widehat{\Ham}(\R^{2n},\omega_0)$ the extension of $\Phi$ to the real line.
To show that the map is injective, we need to notice that for all $t \in \R$, $\gamma(\widehat{\Phi}(t))=\gamma(S_{\alpha t})$ vanishes if and only if $S_{\alpha t}=Id$ that is if and only if $t=0$.

\end{proof}

\begin{prop}\label{prop: the image is closed for gamma}
The image of the extension $\widehat{\Phi}$ of Proposition \ref{prop: gamma continuity in general case} is closed for the $\gamma$-topology inside $\widehat{\Ham}(\R^{2n},\omega_0)$.
\end{prop}

\begin{proof}
Let $\varphi \in \widehat{\Ham}(\R^{2n},\omega_0)$ be a point in the $\gamma$-closure of $\text{Im}(\widehat{\Phi})$. We show that $\varphi$ is in $\text{Im}(\widehat{\Phi})$. Since $\widehat{\Ham(\mathbb R^{2n},\omega)}$ is a metric space, there exists a sequence $x_n \in \R$ such that the sequence $(\widehat{\Phi}(x_n))_n$ $\gamma$-converges to $\varphi$. Hence, the sequence $(\widehat{\Phi}(x_n))_n$ is a Cauchy sequence and therefore
\[
\gamma(\widehat{\Phi}(x_n) (\widehat{\Phi}(x_m))^{-1}) \xrightarrow[n,m \to +\infty]{\gamma} 0.
\]
Using the fact that we have a group homomorphism, we get
\[\gamma(\widehat{\Phi}(x_n-x_m))\xrightarrow[n,m \to +\infty]{\gamma}0.\]
However, as shown above, $\gamma(\widehat{\Phi}(x_n-x_m))=\gamma(S_{\alpha(x_n-x_m)})$. Since $S$ is generated by a positive Hamiltonian in our construction (see the beginning of Section \ref{sec: proof of prop non-aut with roots}), the map $t \mapsto \gamma(S_{\alpha t})$ increases on $\mathbb{R}_{\geq 0}$. This implies that
\[x_n-x_m \xrightarrow[n,m \to +\infty]{} 0,\]
in other words the sequence $(x_n)_n$ is a Cauchy sequence itself. In particular, $(x_n)_n$ converges to $x \in \R$ and by $\gamma$-continuity,
\[\widehat{\Phi}(x_n) \xrightarrow[n\to +\infty]{\gamma} \widehat{\Phi}(x)=\varphi.\]
Thus, $\varphi$ is indeed in $\text{Im}(\widehat{\Phi})$.
\end{proof}

We ask the following natural question.

\begin{ques}\label{ques: honest maps in the extension}
Can we associate elements of $\text{Im}(\widehat{\Phi})$ with honest maps? Are they homeomorphisms? Are they diffeomorphisms?
\end{ques}

\subsection{The special case of manifolds with a Hamiltonian circle action}

In this section, we restrict our attention to the case where the symplectic manifold $(M,\omega)$ admits a Hamiltonian circle action $S \colon S^1 \to \Ham(M,\omega)$. In this setting, we can adapt Proposition \ref{prop: gamma norm} and Proposition \ref{prop: gamma continuity in general case} to prove the following. Note that the $\gamma$-norm can be defined on every symplectic manifold using Floer homology by a work of Fukaya and Ono, \cite{FuOn99}. In general, the $\gamma$-norm enjoys the same properties as the one defined by Viterbo, except for property (v) of Proposition \ref{prop: gamma norm}.

\begin{prop}\label{prop: gamma continuity in toric case}
Let $(M,\omega)$ be a symplectic manifold with a Hamiltonian circle action $S$. Then, the map 
\[\Phi \colon \mathbb{Q} \to \Ham(M,\omega)\]
as constructed in Theorem \ref{thm: Q to Ham} is $\gamma$-continuous. As a consequence, $\Phi$ extends as a homomorphism in the Humilière-completion
\[\widehat{\Phi} \colon \mathbb R \to \widehat{\Ham}(M,\omega).\]
Moreover, the map $\widehat{\Phi}$ is $1/\alpha$-periodic for an irrational number $\alpha\in \mathbb R$ and descends to a map, still denoted $\widehat{\Phi}$,
\[\widehat{\Phi} \colon S^1\to \widehat{\Ham}(M,\omega).\]
This last map is injective and its image is closed for the $\gamma$-topology.
\end{prop}

\begin{proof}
As in the proof of Proposition \ref{prop: gamma continuity in general case}, we have that
\[\gamma(\Phi(r))=\gamma(S_{\alpha r}).\]
 Since the map $t \mapsto S_{\alpha t}$ is $1/\alpha$-periodic and $\widehat{\Phi}$ is a group homomorphism, we can conclude that $\widehat{\Phi}$ is also $1/\alpha$-periodic by the non-degeneracy of the $\gamma$-norm. 

To show that the image is closed, we adapt the proof of Proposition \ref{prop: the image is closed for gamma}  by noticing that the map $t \mapsto \gamma(S_{\alpha t})$ is $1/\alpha$-periodic, continuous and takes the value 0 only at integer multiples of $\frac{1}{\alpha}$. Hence, a sequence of maps $(\widehat{\Phi}(x_n))_n$ is a Cauchy sequence for the $\gamma$-topology implies that the sequence $(x_n)_n$ is also Cauchy in $S^1$, we conclude the proof as in Proposition \ref{prop: the image is closed for gamma}.
\end{proof}

We are still unable to answer Question \ref{ques: honest maps in the extension} in this setting, but we have the following partial answer for $(\mathbb CP^n,\omega_{FS})$ the standard symplectic projective space. We recall that in this case, the $\gamma$-norm is known to be continuous for the $C^0$-topology; Kawamoto has proved this fact in \cite{Kaw22}.

\begin{prop}\label{prop: uncountable elts that are homeos}
Let
\[\widehat{\Phi} \colon S^1\to \widehat{\Ham}(\mathbb CP^n,\omega_{FS})\]
be the extension defined in Proposition \ref{prop: gamma continuity in toric case} and assume that the rotation number of $\Phi(1)$ is exponentially Liouville, then there exists an uncountable subgroup $G \subset S^1$ such that all the elements in $\widehat{\Phi}(G)$ are all homeomorphisms.
\end{prop}

In the construction of the map $\Phi$ of Theorem \ref{thm: Q to Ham}, we say that the number $\alpha \in \R$ that is the limit of the sequence $(\alpha_n)$ is the \textit{rotation number} of the Hamiltonian diffeomorphism $\Phi(1)$. We also say that a real number $\alpha$ is \textit{exponentially Liouville} if
\[\forall c >0, \exists k,n \in \Z \text{ s.t. } 0< \vert k\alpha-n\vert < e^{-ck} \]
The set of exponentially Liouville numbers is residual and hence non-empty; see \cite{GG18}.

To prove Proposition \ref{prop: uncountable elts that are homeos}, we need the following two lemmas and one corollary. We postpone their proof for the end of this section.

\begin{lemma}\label{lemma: complete metric no isolated is uncountable}
    A complete metric space with no isolated points is uncountable.
\end{lemma}

Joksimovic and Seyfaddini have proved the following theorem in \cite{JS24}. We recall first that a map $\varphi$ is said to be \textit{$C^0$-rigid} if there exists a sequence of positive integers $n_k$ such that
\[\varphi^{n_k} \xrightarrow[k\to +\infty]{C^0}Id.\]

\begin{thmA}[Joksimovic-Seyfaddini, \cite{JS24}]\label{thm: cinfty rigidity of AKPR}
Let $(M,\omega)$ be a symplectic manifold with a Hamiltonian circle action $S$ and $\Phi$ the map constructed above. Assume that the rotation number of $\Phi(1)$ is exponentially Liouville. Then, the map $\Phi(1)$ is $C^0$-rigid.
\end{thmA}

This Theorem has the following implication.

\begin{coro}\label{coro: our map has no isolated points}
Under the same assumption as in Theorem \ref{thm: cinfty rigidity of AKPR}, the set $\text{Im}(\Phi)\in \Ham(M,\omega)$ has no isolated point for the $C^0$-topology.
\end{coro}

\begin{proof}[Proof of Corollary \ref{coro: our map has no isolated points}]
From Theorem \ref{thm: cinfty rigidity of AKPR}, we get $(n_k)_k$ a sequence of integers such that for the $C^0$-topology $\lim_{k\to \infty}\Phi(n_k)=0$, this implies that $Id=\Phi(0)$ is not isolated, since $\Phi$ is a group homomorphism, then for every $q \in \mathbb Q$, the point $\Phi(q)$ is also not isolated. This finishes the proof of the corollary.
\end{proof}

We can now prove Proposition \ref{prop: uncountable elts that are homeos}.

\begin{proof}[Proof of Proposition \ref{prop: uncountable elts that are homeos}]
Since $\text{Im}(\Phi)$ is an abelian subgroup of $\Ham(\mathbb CP^n,\omega_{FS})$ and using the $C^0$-continuity of the $\gamma$-norm, its closure for the $C^0$-topology is also an abelian subgroup of $\widehat{\Ham}(\mathbb CP^n,\omega)$. We have the following inclusions
\[\overline{\text{Im}(\Phi)}^{C^0}\subset \text{Im}(\widehat{\Phi}) \subset \widehat{\Ham}(\mathbb CP^n,\omega).\]
From Proposition \ref{prop: gamma continuity in toric case}, we get that 
\[G:=\widehat{\Phi}^{-1}\left(\overline{\text{Im}(\Phi)}^{C^0}\right)\]
is a subgroup of $S^1$. From Corollary \ref{coro: our map has no isolated points}, we get that the $C^0$-closure of $\text{Im}(\Phi)$ has no isolated points, and from Lemma \ref{lemma: complete metric no isolated is uncountable} it is uncountable. This means that $G$ is also uncountable, and since its image under $\widehat{\Phi}$ lies in $\Ham(\mathbb CP^n,\omega)$, the elements of the image are true homeomorphisms.
\end{proof}

If one can show the $C^\infty$-rigidity of maps obtained by the Anosov-Katok method with exponentially Liouville rotation number, one could extend Proposition \ref{prop: uncountable elts that are homeos} to additionally find an uncountable number of diffeomorphisms in $\text{Im}(\widehat{\Phi})$. We now prove Lemma \ref{lemma: complete metric no isolated is uncountable}.

\begin{proof}[Proof of Lemma \ref{lemma: complete metric no isolated is uncountable}]
Let $(X,d)$ be a complete metric space with no isolated points. We construct $2^\omega$ Cauchy sequences that converge to pairwise different points. We inductively construct sets $A_n$ of $2^n$ pairwise disjoint balls as follows. Let $B$ be a ball of radius $\varepsilon$. Since no isolated points exist, we find $B_1,B_2\subset B$ two non-empty disjoint balls with a radius smaller than $\varepsilon \over 2$. We define 
$$A_{n+1} = \bigcup_{B \in A_n}\{B_1,B_2\}$$
while $A_0 =\{B\}$ is a singleton of an arbitrary ball. For any decreasing sequence 
$$E_0\supset E_2\supset ...\supset E_n\supset ...$$
of balls in $A = \bigcup_{n\in \mathbb Z_\geq 0}A_n$ we have that $E =\bigcap_{n\in \mathbb Z_{\geq 0}}E_n$ is non-empty (since $X$ is complete). There are $2^\omega$ such sequences, which all lead to pairwise disjoint sets $E$. This finished the proof.
\end{proof} 

\bibliographystyle{alpha}
\bibliography{biblio}

\end{document}